\documentclass[times,doublespace]{amsart}
\usepackage{amsfonts,amssymb,amscd}
\usepackage{amsxtra,xy}
\usepackage{epsfig,amssymb}
\usepackage{mathrsfs}
\usepackage{epsfig}
\usepackage{bm} 
\usepackage{hyperref}

\usepackage[OT2,T1]{fontenc} 

\DeclareSymbolFont{cyrletters}{OT2}{wncyr}{m}{n}
\DeclareMathSymbol{\sha}{\mathalpha}{cyrletters}{"58}
\DeclareMathSymbol{\tha}{\mathalpha}{cyrletters}{"51}

\def\krn{{\rm{Ker}}\,}
\def\cok{{\rm{Coker}}\,}
\def\img{{\rm{Im}}\e}

\newcommand{\s}{\mathscr }

\newcommand{\res}{{\rm res}}
\newcommand{\trans}{{\rm trans}}
\newcommand{\tra}{{\rm tr}}

\newcommand{\Z}{{\mathbb Z}}
\newcommand{\Q}{{\mathbb Q}}

\newcommand{\spec}{\mathrm{ Spec}\,}
\def\le{\kern 0.03em}
\newcommand{\bg}{{\mathbb G}}

\def\fs{F^{\e\rm{s}}}

\newcommand\be{\kern -.1em}

\newcommand\lbe{\kern -.025em}
\newcommand\e{\kern 0.08em}

\newcommand\g{\varGamma}

\newcommand\gv{\g_{\be w}}

\newcommand{\kwv}{K_{\lbe w}}

\newcommand{\kwvf}{\kwv\le/F_{\be v}}

\newcommand{\lra}{\longrightarrow}

\newcommand{\Hom}{ \mathrm{Hom}} 

\newtheorem{lemma}{Lemma}[section]
\newtheorem{theorem}[lemma]{Theorem}

\newtheorem{proposition}[lemma]{Proposition}
\theoremstyle{definition}

\theoremstyle{remark}
\newtheorem{remark}[lemma]{Remark}
\newtheorem{remarks}[lemma]{Remarks}

\begin{document}

\title[On Tate-Shafarevich groups of 1-motives]{On Tate-Shafarevich groups of 1-motives} 

\author{Cristian D. Gonz\'alez-Avil\'es}
\address{Departamento de Matem\'aticas, Universidad de La Serena, Cisternas 1200, La Serena 1700000, Chile}
\email{cgonzalez@userena.cl}
\thanks{The author was partially supported by Fondecyt grants 1120003 and 1160004.}
\date{\today}     

\begin{abstract} We establish the finiteness of the kernel and cokernel of the restriction map $\res^{\le i}\colon\sha^{\le i}\lbe(F,M\e)\to \sha^{\le i}\lbe(K,M\e)^{\g}$ for $i=1$ and $2$, where $M$ is a (Deligne) 1-motive over a global field $F$ and $K\be/\lbe F$ is a finite  Galois extension of global fields with Galois group $\g$.
\end{abstract}

\maketitle 

\input xy      
\xyoption{all}  

\section{Introduction}

Let $F$ be a global field with fixed separable algebraic closure $\fs$ and corresponding absolute Galois group $\g_{\! F}$. A (Deligne) 1-motive over $F$ is a complex in degrees $-1$ and $0$ of commutative $F$-group schemes $M=[\e Y\to G\,]$, where $Y$ is \'etale-locally isomorphic to $\Z^{r}$ for some integer $r\geq 0$ and $G$ is a semiabelian variety over $F$ given as an extension $0\to T\to G\to A\to 0$, where $A$ is an abelian variety over $F$ and $T$ is an $F$-torus. The Tate-Shafarevich group in degree $i$ of $M$, where $i=1$ or $2$, is 
by definition
\begin{equation}\label{fi}
\sha^{\e i}\lbe(F,M\e)=\krn\!\be\left[\e H^{\le i}(F,
M\e)\to\displaystyle\prod_{\text{all $v$}}\! H^{\e i}(F_{\be v},M\e)\right],
\end{equation}
where the indicated map is a natural localization map. The cohomology groups appearing above are flat (fppf) cohomology groups. However, since $Y$ and $G$ are smooth and commutative $F$-group schemes, a standard five-lemma argument applied to a certain diagram enables us to identify, when convenient, the fppf cohomology groups $H^{\e i}(F,M\le)$ and the Galois cohomology groups $H^{\e i}(\g_{\be F}, M(\lbe\fs))$. Consequently, the above groups, and all other groups appearing in this paper which are closely related to them (e.g., those in Theorem \eqref{main1} below), can be computed using either flat or Galois cohomology. It is well-known that the groups \eqref{fi} are objects of great arithmetical interest which are notoriously difficult to study. For example, if $M=A$ is an abelian variety over $F$, then $\sha^{\e 1}\be(F,A\e)$ encodes important arithmetical information about $A$ but even the finiteness of this group is a major outstanding conjecture in number theory. Regarding the latter group, the author \cite{ga00} and Yu \cite{yu} studied the restriction map $\res^{\le 1}\colon\sha^{\le 1}(F,A\e)\to \sha^{\le 1}(K,A\e)^{\g}$ associated to a finite Galois extension of number fields $K\be/\lbe F$ with Galois group $\g$ and used that study to relate the orders (assumed finite) of $\sha^{\le 1}(F,A\e)$ and $\sha^{\le 1}(F,A^{\chi}\e)$, where $A^{\chi}$ is a quadratic twist of $A$. In this paper we extend the methods of \cite{yu} to study, independently of any finiteness hypotheses, the restriction map $\res^{\le i}\colon\sha^{\le i}\lbe(F,M\e)\to \sha^{\le i}\lbe(K,M\e)^{\le\g}$ for $i=1$ and $2$ and arbitrary 1-motives $M$ over arbitrary global fields $F$. We obtain the following finiteness result.

\begin{theorem}\label{main1} Let $M$ be a $1$-motive over a global field $F$ and let
$K\be/\lbe F$ be a finite Galois extension of global fields with Galois group $\g$.
Then, for $i=1$ and $2$, the kernel and cokernel of the restriction map 
\[
\res^{\le i}\colon\sha^{\le i}\lbe(F,M\e)\to \sha^{\le i}\lbe(K,M\e)^{\g}
\]
are finite groups annihilated by a power of $[\e K\lbe\colon \! F\,]$.
\end{theorem} 

\smallskip

We should note that, when $M=G$ is a semiabelian variety (and thus an algebraic group), the proof of that part of the theorem which refers to the kernel of the indicated map is not difficult and follows from a standard restriction-corestriction argument. However, the fact that $M$ is, in general, only a complex of algebraic groups makes the proof of the theorem rather involved for both the kernel and cokernel of $\res^{\le i}$.

\smallskip

The proof of the theorem will show that, in fact, $\krn\be(\res^{\le i}\le)$ is annihilated by $[\e K\lbe\colon \! F\,]$ (see Remark \ref{last}). Regarding the case $i=1$ of the theorem, we note that both the source and the target of $\res^{1}$ are finite groups if (as widely expected) $\sha^{\le 1}\lbe(K,A\e)$ is a finite group, where $A$ is the abelian variety part of $M$. However, when $i=2$, Jossen \cite[\S7]{joss} has constructed an example of a semiabelian variety $G$ over $\Q$ such that the group $\sha^{\e 2}\lbe (\Q,G\e)$ is infinite. 

We briefly indicate the contents of each Section. Section \ref{pre} consists of preliminaries. Section \ref{1-mot} discusses 1-motives over fields (e.g., global and local fields). Section \ref{gr} establishes that the primary components of Tate-Shafarevich groups of 1-motives of global fields in cohomological degrees 1 and 2 are groups of finite cotype. This result is new 
in cohomological degree 2 over global function fields and is well-known in the remaining cases, although no formal statement to this effect seems to have appeared in print before. Section \ref{mai1} contains the proof of Theorem \ref{main1}. The Appendix \ref{app} repairs an error which appears in a reference relevant to this paper.

\section*{Acknowledgements}
I thank Alessandra Bertapelle for pointing out an error in a preliminary version of this paper and for suggesting a number of improvements to the presentation, as well as K\c{e}stutis \v{C}esnavi\v{c}ius for pointing out, and helping me fix, a (fortunately inconsequential) error in \cite[proof of Proposition 4.10]{ga09}. See the Appendix. I am also grateful to the referee of this paper for agreeing to review it and for pointing out the need to clarify certain parts of it. Finally, I thank Chile's Fondecyt for financial support via research grants 1120003 and 1160004.

\section{Preliminaries}\label{pre}

For any abelian group $B$ and positive integer $n$, we will write
$B_{n}$ for the $n$-torsion subgroup of $B$ and $B/n$ for $B/nB$. If $B$ is finite, its order will be denoted by $|\le B\le|$. Let $\ell$ be a prime number. We will write $B(\e\ell\e)$ for the $\ell$-primary torsion subgroup of $B$ and $B^{(\e\ell\e)}$ for $\varprojlim_{\e m}(B/\e\ell^{\e m})$.

If $B$ is a topological abelian group, we will write
$B^{D}=\text{Hom}_{\le\text{cont.}}(B,\Q/\Z)$ and endow it with the
compact-open topology, where $\Q/\Z$ is given the discrete
topology. Via Pontryagin duality, the category of profinite abelian groups is anti-equivalent to the category of discrete and torsion abelian groups. Both are abelian categories. Further, if $0\to A\to B\to C\to 0$ is an exact sequence of discrete and torsion abelian groups, then the dual sequence $0\to C^{D}\to B^{D}\to A^{D}\to 0$ is an exact sequence of profinite abelian groups.

\begin{lemma}\label{fincot} Let $\ell$ be a prime number and let $B$ be a discrete $\ell$-primary torsion abelian group. Then the following conditions are equivalent.
\begin{itemize}
\item[(i)] The profinite abelian group $B^{D}$ is a finitely generated $\Z_{\e\ell}$-module.
\item[(ii)] $B$ is isomorphic to $(\Q_{\le\ell}/\le\Z_{\le\ell})^{r}\oplus A$, where $r\geq 0$ is an integer and $A$ is a finite abelian $\ell$-group.
\item[(iii)] $B_{\e\ell}$ is a finite group.
\end{itemize}
\end{lemma}
\begin{proof} Assertion (ii) follows from (i) and duality by noting that $\Z_{\e\ell}^{D}=\Q_{\le\ell}/\le\Z_{\le\ell}$. The implication (ii) $\!\!\implies\!\!$ (iii) is trivial. Finally, if (iii) holds, then $B^{D}\!/\ell=(B_{\ell})^{D}$ is finite and (i) holds.
\end{proof}

If $\ell$ is a prime number, a discrete $\ell$-primary torsion abelian group $B$ which satisfies the equivalent conditions of the previous lemma is said to be {\it of finite cotype}. If $B$ as above is of finite cotype, then every subgroup and every quotient of $B$ is a group of finite cotype.

\begin{lemma}\label{es2} Let $B$ be a discrete and torsion abelian group such that $B(\le\ell\e)$ is a group of finite cotype for every prime number $\ell$. Then every subgroup and every quotient of $B$ of finite exponent is, in fact, finite.
\end{lemma}
\begin{proof} If $A$ is a subgroup of $B$ and $n$ is an integer such that $n\le A=0$, then $A\subset B_{n}$ is finite since $B_{n}$ is finite. If $C$ is a quotient of $B$ such that $n\le C=0$, then $C^{D}\subset (B^{D})_{n}=\prod_{\,\ell\le|n}(B(\le\ell\e)^{D})_{n}$, which is finite since each group $B(\le\ell\e)^{D}$ is a finitely generated $\Z_{\e\ell}$-module.
This proves the lemma.
\end{proof}

\begin{proposition}\label{ker-cok} Let $A\overset{\!\be f}\to B\overset{\!g}\to C$ be a pair of morphisms in an abelian category $\mathcal A$. Then there exists an induced exact sequence in $\mathcal A$ 
\[
0\to\krn f\to\krn(\e g\be\circ\! f\e)\to\krn g\to\cok f\to\cok\le(\e g\be\circ\! f\e)\to\cok g\to 0.
\]
\end{proposition}
\begin{proof} See, for example, \cite[Hilfssatz 5.5.2, p.~45]{bp}.
\end{proof}

\begin{proposition}\label{wk} If $i>0$, $\g$ is a finite group and $A$ is a $\g$-module, then $H^{\le i}(\g,A)$ is annihilated by $|\g\e|$. If, in addition, $A$ is finitely generated, then $H^{\le i}(\g,A)$ is finite.
\end{proposition}
\begin{proof} This is well-known. See, for example, \cite[VIII, \S2, Corollaries 1 and 2, p.~130]{self}.
\end{proof}

\begin{lemma}\label{f-exp} Let $\g$ be a finite group. If $A\to B\to C$ is an exact sequence of $\g$-modules, then the homology of the induced complex of abelian groups $A^{\lbe\g}\to B^{\g}\to C^{\g}$ is annihilated by $|\g\e|$. 
\end{lemma}
\begin{proof} Let $f\colon A\to B$ and $g\colon B\to C$ be the given maps. The exact sequence of $\g$-modules $0\to\krn f\to A\to \img f\to 0$ yields an exact sequence of abelian groups 
\begin{equation}\label{usfl}
0\to (\krn f\le)^{\g}\to A^{\g}\overset{\!h}{\to}(\img f\e )^{\g}\to H^{1}(\g, \krn f\e)
\end{equation}
and $f^{\g}\be\colon\lbe A^{\g}\to B^{\g}$ factors as $A^{\g}\overset{\!h}{\to}(\img f\e )^{\g}\hookrightarrow  B^{\g}$. Thus there exist isomorphisms of abelian groups $\krn\!\left(f^{\g}\e\right)=(\krn f\le)^{\g}$  and $\img\!\left(f^{\g}\e\right)=\img\, h$. Similarly, $\krn\!\be\left(\le g^{\le \g}\e\right)=(\krn g\le)^{\g}=(\img f\e )^{\g}$. Thus the homology of the complex $A^{\g}\to B^{\g}\to C^{\g}$ equals
\[
\krn\!\be\left(\le g^{\le \g}\e\right)\be/\e\img\!\be\left(f^{\g}\e\right)=(\img f\e )^{\g}\be/\e\img\, h=\cok h.
\]
Since $\cok h$ injects into $H^{1}(\g, \krn f\e)$ by \eqref{usfl} and the latter group is annihilated by $|\g\e|$ by Proposition \ref{wk}, the lemma follows.
\end{proof}

We will write ${\bf Ab}$ for the category of abelian groups.

\begin{lemma}\label{wk} Let $n$ be a positive integer and let $0\to A\to B\to C\to 0$ be a short exact sequence in ${\bf Ab}$. Then there exists an induced exact sequence in ${\bf Ab}$
\[
0\to A_{n}\to B_{n}\to C_{n}\to A/n\to B/n\to C/n\to 0.
\]	
\end{lemma}
\begin{proof} This follows by applying the functor $\textrm{Tor}_{*}^{\Z}(\Z/\e n,-)$ to the given short exact sequence using \cite[Calculation 3.1.1, p.~66]{wei}.
\end{proof}

 Let $\g$ be a finite group and let $\s C$ be the quotient category of ${\bf Ab}$ modulo its thick  subcategory of groups of finite exponent annihilated by a power of $|\g\e|$. If $f\colon A\to B$ is a homomorphism of abelian groups, then $f$ is a monomorphism (respectively, epimorphism, isomorphism) in $\s C$ if, and only if, $\krn f$ (respectively, $\cok f$, both $\krn f$ and $\cok f\e$) are groups of finite exponent annihilated by a power of $|\g\e|$ \cite[III, \S1, Lemma 2, p.~366]{gab62}.

\smallskip

\begin{proposition}\label{ref} Let $\mathcal G$ be a profinite group and let $\mathcal N$ be an open and normal subgroup of $\mathcal G$ such that the quotient group $\g=\mathcal G/\e\mathcal N$ is finite. Then, for every continuous $\mathcal G$-module $C$ and every integer $i\geq 1$, the restriction map $\res^{i}\colon H^{\le i}(\mathcal G,C\e)\to H^{\le i}(\e\mathcal N,C\e)^{\g}$ is an isomorphism in $\s C$. In other words, the kernel and cokernel of $\res^{i}$ are groups of finite exponent annihilated by a power of $|\g\le|$.
\end{proposition}
\begin{proof} Denote by ${\bf Ab}_{\le\mathcal G}$ the (abelian) category of abelian groups equipped with a continuous $\mathcal G$-action and equivariant maps, and define ${\bf Ab}_{\lbe\g}$ similarly. Write $\s C_{\mathcal G}$ for the localisation of the abelian category ${\bf Ab}_{\le\mathcal G}$ with respect to the thick subcategory of modules of finite exponent annihilated by a power of $|\g\le|$ and let $\s C_{\g}$ be the corresponding object associated to ${\bf Ab}_{\g}$. Consider the following commutative diagram of categories and functors:
\[
\xymatrix{{\bf Ab}_{\le\mathcal G}\ar[d]\ar[r]^{(-)^{\le\mathcal N}}& {\bf Ab}_{\g}\ar[d]\ar[r]^{(-)^{\g}}& {\bf Ab}\ar[d]\\
\s C_{\mathcal G}\ar[r]^{(-)^{\le\mathcal N}}& \s C_{\g}\ar[r]^{(-)^{\g}}& \s C.}
\]
The vertical arrows are the canonical localisation functors, which are exact. The horizontal compositions are the fixed point functors $(-)^{\mathcal G}$. The classical Hochschild-Serre spectral sequence is the Grothendieck spectral sequence of the composite functor ${\bf Ab}_{\le\mathcal G}\to{\bf Ab}_{\g}\to {\bf Ab}$.
Since the vertical functors are exact, this spectral sequence maps to the Grothendieck spectral sequence of the composite $\s C_{\mathcal G}\to \s C_{\g}\to \s C$. Now, since the functor $(-)^{\g}\colon \s C_{\g}\to\s C$ is exact by Lemma \ref{f-exp}\,, the localised spectral sequence degenerates at the
second stage and we obtain isomorphisms in $\s C$
\[
H^{\le i}(\mathcal G,C\le)\simeq H^{\e i}(\e\mathcal N,C\le)^{\g}.
\]
The preceding morphism is represented in ${\bf Ab}$ by the edge morphism $E^{\e i}\to
E_{2}^{\e 0, i}$, which coincides with the restriction map $\res^{i}$. The proposition is now clear.
\end{proof}

\section{1-motives over a field}\label{1-mot}

Let $F$ be a field and let $\fs$ be a fixed separable algebraic closure of $F$. For any subextension $L/F$ of $\fs\be/F$, we will write $\g_{\be L}$ for $\textrm{Gal}(\e\fs/L\e)$. A (Deligne) 1-motive over $F$ is a complex in degrees $-1$ and 0 of commutative $F$-group schemes $M=[\e Y\be\overset{\!\be u}\to G\,]$, where $Y$ is \'etale-locally isomorphic to $\Z^{r}$ for some integer $r\geq 0$, $G$ is a
semiabelian variety over $F$ given as an extension of an abelian variety $A$ by a torus $T$
\begin{equation}\label{tga}
0\to T\to G\to A\to 0
\end{equation}
and $u$ is a morphism of $F$-group schemes. We will identify the 1-motive $M=[\e 0\to G\,]$ with $G$ placed in degree 0. The exact sequence of complexes
\begin{equation}\label{tga2}
0\to [\e 0\to G\,]\overset{\! \iota}\to[\e Y\to G\,]\to [\e Y\to 0\,]\to 0
\end{equation}
induces the following distinguished triangle in the derived category of the category of bounded complexes of abelian sheaves on the small fppf site over $\spec F$:
\begin{equation}\label{ygm}
Y\to G\overset{\! \iota}\to M\to Y[1].
\end{equation}
We will write $M^{\le *}$ for the 1-motive dual to $M$. Thus
\[
M^{\le *}=[\e Y^{ *}\!\!\overset{u^{\be *}}{\lra} G^{\e *}\e],
\]
where $Y^{ *}$ is the group of characters of $T$ and $G^{\e *}$ is an extension of the abelian variety $A^{\le *}$ dual to $A$ by an $F$-torus $T^{\e *}$ whose group of characters is $Y$. The dual 1-motive comes equipped with a duality pairing, i.e., a morphism in the derived category of fppf sheaves $b\colon M\!\otimes^{\mathbf L}\!M^{\le *}\!\to\bg_{m,F}[1]$ which induces isomorphisms $M\overset{\!\sim}{\to} \mathbf R\Hom(M^{\le *}, \bg_{m,F}[1])$ and $M^{\le *}\overset{\!\sim}{\to}\mathbf R\Hom(M,\bg_{m,F}[1])$. The pairing $b$ induces pairings of fppf cohomology groups
\begin{equation}\label{pings}
H^{\e i}(F, M)\times H^{\e j}(F, M^{\le*})\overset{\!\cup}{\to}H^{\e i+j}(F, M\!\otimes^{\mathbf L}\!M^{\le*})\overset{\!H^{i+j}\be(b)}{\lra}H^{\e i+j+1}(F,\bg_{m}).
\end{equation}
The above constructions can be extended to any base scheme (in place of a field) and, in particular, the notions of a 1-motive over a Dedekind scheme and its dual are defined. Such objects will briefly appear in Remark \ref{spy}(b) and in the proof of Lemma \ref{fcot}. See \cite[\S1]{hs1} and \cite[\S3]{ga09} for more details.

Note that, by \eqref{ygm},
\begin{equation}\label{h-1}
H^{\e -1}(K,M\le)=\krn[\e H^{\le 0}(K, Y\le)\to H^{\le 0}(K, G\le)\e],
\end{equation}
which is a finitely generated abelian group.

\smallskip

Now let $K/F$ be a finite Galois subextension of $\fs\be/F$ with Galois group $\g$.
For every integer $i\geq -1$, set
\begin{equation}\label{hlg}
H^{\e i}(K/F, M\e)=\krn\!\be\left[\e H^{\e i}(F, M\le)
\overset{\!\res^{\lbe i}}{\lra} H^{\e i}(K, M\le)^{\g}\,\right]
\end{equation}
and
\begin{equation}\label{hlgc}
\quad\,\,{\rm C}^{\e i}\lbe(K/F, M\e)=\cok\!\be\left[\e H^{\e i}(F, M\le)\overset{\!\res^{\lbe i}}{\lra} H^{\e i}(K, M\le)^{\g}\,\right]\lbe.
\end{equation}

\begin{lemma}\label{hi0} For every integer $i\geq 1$, $C^{\e i}(K/F, M\e)$ \eqref{hlgc} is a group of finite exponent annihilated by a power of $[\e K\lbe\colon \be F\,]$.
\end{lemma}
\begin{proof} We work in the category $\s C$ of abelian groups modulo groups of finite exponent annihilated by a power of $|\g\e|$. By Lemma \ref{f-exp}, the triangle \eqref{ygm} induces a 5-column exact and commutative diagram in $\s C$
\[
\hspace{-4.1cm}\xymatrix{H^{\e i-1}(F, Y[1]\e)\ar[d]\ar[r]& H^{\le i}(F, G\e)\ar[d]\ar[r]& H^{\le i}(F,M\e)\ar[d]\dots\\
	H^{\e i-1}(K, Y[1]\e)^{\g}\ar[r]& H^{\le i}(K, G\e)^{\g}\ar[r]& H^{\le i}(K,M\e)^{\g}\dots}
\]
\[
\xymatrix{&&&&&\ar[r]& H^{\le i}(F, Y[1]\e)\ar[d]\ar[r]& H^{\le i+1}(F, G\e)\ar[d]\\
&&&&&\ar[r]& H^{\le i}(K, Y[1]\e)^{\g}\ar[r]& H^{\le i+1}(K, G\e)^{\g}.}
\]
By Proposition \ref{ref}, the first, second, fourth and fifth vertical maps above are isomorphisms in $\s C$. Now the five-lemma completes the proof.
\end{proof}

In the case of \eqref{hlg}, the following more precise result holds.

\begin{lemma}\label{hi} For every integer $i\geq -1$, $H^{\e i}(K/F, M\e)$ \eqref{hlg} is a group of finite exponent annihilated by $[\e K\lbe\colon \be F\,]$.
\end{lemma}
\begin{proof} We use a restriction-corestriction argument. The Weil restriction $R_{K/F}(\le Y_{\lbe K})$ is \'etale-locally isomorphic to $\Z^{\le r[\e K\lbe\colon \! F\,]}$. Further, $R_{K/F}(G_{\be K})$ is a semiabelian variety over $F$ by \cite[Corollary A.5.4(3), p.~508]{cgp}. The Weil restriction of $M_{K}$ relative to $K/F$ is the 1-motive over $F$
\[
R_{K\be/F}(M_{\lbe K})=[\e R_{K/F}(\le Y_{K})\overset{\!v}{\to} R_{K/F}(G_{\be K})\e],
\]
where $v=R_{K/F}(u_{K})$. Now, for every $F$-group scheme $H$, there exists a canonical norm (or corestriction) morphism $R_{K\be/F}(\le H_{\be K}\be)\to H$ defined as follows: if $N_{K\be/F}\colon K\to F$ is the norm map associated to the Galois extension $K/F$ and $X$ is any $F$-scheme, then $R_{K\be/F}(\le H_{\be K}\be)(X)\to H(X)$ is the homomorphism $H(X_{\be K})\to H(X)$ induced by $X\times_{\spec F}\spec\be(N_{K\be/F}\be)\colon\, X\to X_{\be K}$. We let $N\colon R_{K\be/F}(M_{\lbe K})\to M$ be the morphism of 1-motives whose components are the norm morphisms $R_{K/F}(\le Y_{K}\be)\to Y$ and $R_{K/F}(\le G_{\be K}\be)\to G$ described above. We now observe that there exists a canonical morphism of $F$-group schemes $Y\!\to R_{K/F}(\le Y_{K})$ which corresponds to the identity morphism of $Y_{\lbe K}$ under the isomorphism 
$\Hom_{F}(\e Y,R_{K/F}(Y_{K}))\simeq\Hom_{K}(\e Y_{\lbe K}, Y_{\lbe K})$ that characterizes the functor $R_{K/F}$. Similarly, there exists a canonical morphism of $F$-group schemes $G\to R_{K/F}(G_{\be K})$. These morphisms are the components of a canonical morphism of 1-motives $j\colon M\to R_{K/F}(M_{\lbe K})$. The composite morphism 
\[
M\overset{\!j}{\to} R_{K/F}(M_{\lbe K}\be)\overset{\!N}{\to} M
\]
is the multiplication-by-$[\e K\lbe\colon \be F\,]$ map on $M$. Now, by the definition of $R_{K\be/F}(M_{\lbe K})$, Shapiro's lemma in Galois cohomology and a standard five-lemma argument (see the proof of the previous lemma), for every integer $i\geq -1$ there exists a canonical isomorphism of Galois hypercohomology groups $e^{\le i}\colon H^{\e i}(F, R_{K/F}(M_{\lbe K})\le)\overset{\!\sim}{\to}H^{\e i}(K, M\le)$. We define the norm map $N^{\le i}\colon H^{\e i}(K, M\le)\to H^{\e i}(F, M\le)$ as the composition
\[
H^{\e i}(K, M\le)\overset{\! (e^{i})^{-1}}{\lra}H^{\e i}(F, R_{K/F}(M_{\lbe K})\le)\overset{H^{\le i}\be(N\le)}{\lra}
H^{\e i}(F, M\le).
\]
The restriction map $\res^{\le i}\colon H^{\e i}(F, M\le)\to H^{\e i}(K, M\le)$ is the composition
\[
H^{\e i}(F, M\le)\overset{\!H^{\le i}\be(\e j\lbe)}{\lra}H^{\e i}(F, R_{K/F}(M_{\lbe K})\le)\overset{\! e^{i}}{\lra}H^{\e i}(K, M\le)
\]
and the composite map
\[
H^{\e i}(F, M\le)\overset{\!\res^{\le i}}{\lra}H^{\e i}(K, M\le)\overset{\! N^{\le i}}{\lra}H^{\e i}(F, M\le)
\]
is the multiplication-by-$[\e K\lbe\colon \be F\,]$ map on $H^{\e i}(F, M\le)$. The lemma follows easily.
\end{proof}

We now assume that $F$ a global field, i.e., $F$ is either a finite extension of $\Q$ (the number field case) or is finitely generated and of transcendence degree 1 over a finite field of constants (the
function field case). We will write $p$ for the characteristic exponent of $F$. Thus $p=1$ if $F$ is a number field and $p={\rm char}\e F$ if $F$ is a function field.
If $v$ is a prime of $F$, $F_{\be v}$ will denote the completion of $F$ at $v$. For each prime $v$ of $F$, we choose and fix a prime $\overline{v}$ of $\fs$ lying above $v$ and write $w$ for the prime of $K$ lying below $\overline{v}$. The decomposition group of $w$ in $\g$ will be denoted by $\gv$ and identified with $\textrm{Gal}(\kwv\le/F_{v})$.

If $v$ is an archimedean prime of $F$ and $i$ is any integer, $H^{\e
i}(F_{\be v}, M\le)$ will denote the $i$-th Tate cohomology group of $M_{\lbe F_{\be v}}$ defined in \cite[p.~103]{hs1}. For every $i\in\Z$, $H^{\e i}(F_{\be v}, M\le)$ is a finite $2$-torsion group.

\begin{lemma}\label{lexp} For every integer $i\geq -1$, $\prod_{\e\rm{all}\,\text{v}}\lbe H^{\le i}\lbe(\kwvf\le, M\le)$ is a group of finite exponent annihilated by $[\e K\lbe\colon\be F\e]$.
\end{lemma}
\begin{proof} The lemma is clear since each group $H^{\e i}(\kwvf\le, M\le)$ \eqref{hlg} is annihilated by $[\e\kwv\lbe\colon\be F_{\be v}\le]$ by Lemma \ref{hi} and $[\e\kwv\lbe\colon\be F_{\be v}\le]$ divides $[\e K\lbe\colon\be F\e]$ for every prime $v$ of $F$.
\end{proof}

\smallskip

\section{Tate-Shafarevich groups}\label{gr}

For every finite set $S$ of primes of $F$ and integer $i$, set\,\footnote{We warn the reader against confusing $\sha^{\le i}_{S}(F,M\e)$ with the similarly-denoted Tate-Shafarevich group associated to the Galois group of the maximal extension of $F$ which is unramified outside of $S$.}
\[
\sha^{\le i}_{S}(F,M\e)=\krn\!\be\left[\e H^{\le i}(F,
M\e)\to\displaystyle\prod_{v\notin S}H^{\e i}(F_{\be v},M\e)\right].
\]
The Tate-Shafarevich group in degree $i$ of $M$ is the group $\sha^{\le i}_{\emptyset}(F,M\e)$. Now let $\sha^{\e i}_{\le\omega}(F, M\le)$ denote the subgroup of $H^{\e i}(F,M\le)$ of all classes which are locally trivial at all but finitely many primes of $F$. Thus
\begin{equation}\label{union}
\sha^{\le i}_{\e\omega}\lbe(F,
M\le)=\displaystyle\bigcup_{S}\e\sha^{\le i}_{S}(F,
M\le)\,\subset\, H^{\le i}(F,M\e),
\end{equation}
where the union extends over all finite sets $S$ of primes of $F$. For $i=1$ and 2, we will write
\begin{equation}\label{sci}
{\rm C}_{\omega}^{\e i}\lbe(K/F, M\e)=\cok\!\be\left[\e\sha_{\omega}^{\le i}(F, M\le)\overset{\!\res^{\lbe i}_{\lbe\omega}}{\lra}\sha_{\omega}^{\le i}(K,M\le)^{\g}\,\right]\lbe.
\end{equation}

\begin{remarks}\label{spy} \indent
\begin{itemize}
\item[(a)] If $M$ is an arbitrary 1-motive, then 
$\sha^{\e 2}_{\le\omega}(F,M\le)=H^{\e 2}(F,M\le)$ by \cite[p.~117, first paragraph]{hs1}. Consequently, ${\rm C}_{\omega}^{\e 2}\lbe(K/F, M\e)=
{\rm C}^{\e 2}\lbe(K/F, M\e)$ \eqref{hlgc}.
\item[(b)] If $\e\sha^{\e 1}(F,A\le)$ is finite, then $\sha^{\e 1}(F,M\le)$ is finite. Indeed, if $\ell\neq p$, then the finiteness of $\sha^{\e 1}(F,M\le)(\ell\le)$ follows from that of $\sha^{\e 1}(F,A\le)(\ell\le)$ by \cite[Proposition 3.7 and proof of Theorem 4.8]{hs1}. See also \cite[Remark 5.10]{hs1} and note that the hypothesis of \cite[Proposition 3.7]{hs1} is satisfied by \cite[Remark I.6.14(c), p.~84]{adt}. Now, if $F$ is a function field with associated curve $C$, then \cite[Lemma 6.5]{ga09} shows that $\sha^{\e 1}(F,M\le)(\e p)=D^{1}(U,\s M\e)(\e p)$ for an appropriate affine, open and dense subscheme $U$ of $C$, where $\s M$ is a 1-motive over $U$ whose generic fiber is $M$ and  $D^{1}(U,\s M\e)=\krn[\e H^{1}(U,\s M\e)\to\bigoplus_{v\notin U}H^{1}(F_{\be v},M\e)]$. The finiteness of $D^{1}(U,\s M\e)(\e p)$ follows from that of $\,\sha^{\e 1}(F,A)(\e p)$, as in the proof of \cite[Corollary 3.7, p.~111]{hs1}. See \cite[proof of Lemma 5.4, p.~222]{ga09}. 
\end{itemize}
\end{remarks}

\begin{lemma}\label{sfex} If $\le i=1$ or $2$, ${\rm C}_{\omega}^{\e i}\lbe(K/F, M\e)$ \eqref{sci} is a group of finite exponent annihilated by a power of $[\e K\lbe\colon\lbe \! F\,]$.
\end{lemma}
\begin{proof} By Remark \ref{spy}(a)\,, ${\rm C}_{\omega}^{\e 2}\lbe(K/F, M\e)=
{\rm C}^{\e 2}\lbe(K/F, M\e)$ and the latter group is  annihilated by a power of $[\e K\lbe\colon\lbe \! F\,]$ by Lemma \ref{hi0}. Now assume that $i=1$, let $S$ be any finite set of primes of $F$ and let $S^{\e\prime}$ be the set of primes of $K$ lying above the primes in $S$. Now, by the semilocal theory of \cite[\S2.1]{ch}, there exist canonical isomorphisms 
\[
\big(\textstyle\prod_{\e w^{\prime}\notin S^{\le\prime}}H^{\le 1}\lbe(K_{w^{\prime}},M\e)\big)^{\g}=
\big(\textstyle\prod_{v\notin S}\textstyle\prod_{\e w^{\prime}|v}H^{\le 1}\lbe(K_{w^{\prime}},M\e)\big)^{\g}\simeq\displaystyle\prod_{v\notin S}H^{\e 1}(K_{w},M\e)^{\gv}
\]
where, for every prime $v\notin S$, $w$ is the prime of $K$ lying above $v$ fixed previously. Identifying the preceding groups via the indicated isomorphism, we obtain a canonical exact and commutative diagram of abelian groups
\[
\xymatrix{0\ar[r]& \sha^{\le 1}_{S}(F,M\e)\ar[d]^{\res_{\lbe S}^{\lbe 1}}\ar[r]& H^{\le 1}(F,M\e)\ar[d]^{\res^{\lbe 1}}\ar[r]& \displaystyle\prod_{v\notin S}H^{\e 1}(F_{\be v},M\e)\ar[d]^{\prod_{v\notin S}\res_{\lbe v}^{\lbe 1}}\\
0\ar[r]&\sha_{S^{\le\prime}}^{\le 1}(K,M\le)^{\g}\ar[r]& H^{\le 1}(K,M\e)^{\g}\ar[r]&
\displaystyle\prod_{v\notin S}H^{\e 1}(K_{w},M\e)^{\gv}.}
\]
By Lemma \ref{hi0}, the cokernel of the middle vertical map is annihilated by a power of $[\e K\lbe\colon\lbe \! F\,]$. On the other hand, by Lemma \ref{lexp}, the kernel of the right-hand vertical map above is annihilated by $[\e K\lbe\colon\lbe \! F\,]$. We conclude that the cokernel of $\res_{\lbe S}^{\le 1}$ is annihilated by a power of $[\e K\lbe\colon\lbe \! F\,]$ {\it which is independent of the choice of $S$}. The lemma follows.
\end{proof}

\begin{lemma}\label{fcot} Let $\ell$ be a prime number. Then $\sha^{\e i}(F, M\le)(\e\ell\,)$ is a group of finite cotype for $i=1$ and $i=2$.
\end{lemma}
\begin{proof} If $\ell\neq p$, there exists a nonempty open affine subscheme $U$ of either the spectrum of the ring of integers of $F$ (in the number field case) or of the curve $C$ associated to $F$ (in the function field case) such that $\ell$ is invertible on $U$, $M$ extends to a 1-motive $\s M$ over $U$ and $\sha^{\e 1}(F, M\le)(\e\ell\e)=D^{\e 1}(U,\s M\e)(\e\ell\e)$, where $D^{\e i}(U,\s M\e)=\krn[\e H^{\e i}(U,\s M\e)\to\bigoplus_{v\notin U}H^{\e i}(F_{\be v},M\e)]$. See \cite[Remark 5.10 and proof of Theorem 4.8]{hs1}. The case $i=1$ and $\ell\neq p$ of the lemma now follows from \cite[Lemma 3.2(2)]{hs1}. For the case $i=1$ and $\ell=p$ of the lemma, see  \cite[Lemma 6.5]{ga09}\,. We now address the case $i=2$. It is shown in
\cite[proof of Proposition 4.12, p.~116]{hs1} that, for every sufficiently small set $U$ as above, there exists an exact sequence of torsion \cite[Lemma 3.2(1)]{hs1} abelian groups
\[
0\to B(U,\s M\e)\to D^{\e 2}(U,\s M\e)\to 
\sha^{\e 2}(F, M\le)\to 0,
\]
where $B(U,\s M\e)$ is defined as the kernel of the map $D^{\e 2}(U,\s M\e)\to 
\sha^{\e 2}(F, M\le)$ induced by the canonical map $H^{\le 2}(U,\s M\e)\to 
H^{\le 2}(F,M\le)$. Now let $m\geq 1$ be an integer and apply Lemma \ref{wk} to the preceding short exact sequence setting $n=\ell^{\e m}$ in that lemma. We obtain an exact sequence of abelian groups $D^{\e 2}(U,\s M\e)_{\e\ell^{m}}\to 
\sha^{\e 2}(F, M\le)_{\e\ell^{m}}\to B(U,\s M\e)\otimes_{\le\Z}\Z/\ell^{\e m}$. Taking direct limits over $m$ above and noting that $\varinjlim_{\e m}\e B(U,\s M\e)\otimes_{\le\Z}\Z/\ell^{\e m}=B(U,\s M\e)\le\otimes_{\e\Z}\e \Q_{\e\ell}/\le\Z_{\e\ell}=0$ since $B(U,\s M\e)$ is torsion and $\Q_{\e\ell}/\le\Z_{\e\ell}$ is divisible, we conclude that $\sha^{\e 2}(F, M\le)(\e\ell\e)$ is a quotient of $D^{\e 2}(U,\s M\e)(\e\ell\e)$ for any prime $\ell$. Now, if $\ell\neq p$ is invertible on $U$, then $D^{\e 2}(U,\s M\e)(\e\ell\e)$ is a group of finite cotype by \cite[Lemma 3.2(2)]{hs1}. Thus it remains only to check that $D^{\e 2}(U,\s M\e)(\e p\e)$ is a group of finite cotype for any set $U$ as above (in the function field case). By \cite[proof of Theorem 5.10, p.~227]{ga09}, there exists a perfect pairing of topological abelian groups
\[
D^{\e 1}(U,T_{\le p}\le(\s M^{\le *}\e)\e)\times D^{\e 2}(U,\s M)(\e p\e)
\to\Q_{\e p}/\le\Z_{\e p},
\]
where $D^{\e 1}(U,T_{\le p}(\s M^{\le *}))=\varprojlim_{\,m} D^{\e 1}(U,T_{\Z/p^{\e m}}(\s M^{\le *}))$. Thus we are reduced to checking that
$D^{\e 1}(U,T_{\le p}\le(\s M^{\le *}\e))$ is a finitely generated $\Z_{\e p}$-module. To this end, we recall from \cite[proof of Lemma 5.8(b), p.~226, line 3]{ga09} the exact sequence
of abelian groups
\[
0\to D^{\e 0,\e(p)}(U,\s M^{\le *}\lbe)\to D^{\e 1}(U,T_{p}(\s M^{\le *}\lbe))\to
T_{p}\,D^{\e 1}(U,\s M^{\le *}),
\]
where $D^{\e 0,\e(p)}(U,\s M^{\le *})=\krn[\,{H}^{\e 0}(U,\s
M^{\le *})^{(\e p)}\to\bigoplus_{v\notin U}{H}^{\e
0}(F_{\be v},M^{\le *})^{(\e p)}\,]$. The right-hand group above is a finitely generated $\Z_{\e p}$-module by Lemma \ref{fincot}\, since $D^{\e 1}(U,\s M^{\le *})_{\le p}$ is finite by \cite[proof of Lemma 5.4, p.~222]{ga09}. To show that  $D^{\e 0,\e(p)}(U,\s M^{\le *})$ is also a finitely generated $\Z_{\e p}$-module, it suffices to check that ${H}^{\e 0}(U,\s M^{\le *})$ is a finitely generated abelian group. The corresponding statement for number fields is established in \cite[Lemma 3.2(3)]{hs1} via a devissage argument that also works for function fields. This completes the proof.
\end{proof}

\section{Proof of Theorem \ref{main1}}\label{mai1}

In this Section we prove Theorem \ref{main1}\,. The main ingredients of the proof are Lemmas \ref{lexp}\,, \ref{sfex} and \ref{fcot}.

For $i=1$ or 2, let
\begin{equation}\label{loc}
\lambda^{\lbe i}(F,M\le)\colon \sha^{\e i}_{\omega}(F,M\le)\to
\bigoplus_{\text{all $v$}}H^{\e i}(F_{\be v},M\le)\e,
\end{equation}
be the canonical localization map whose $v$-component is induced by the restriction map ${H}^{\e i}(F,M\le)\to H^{\e i}(F_{\be v},M\le)$. There exists a canonical exact sequence of discrete and torsion abelian groups
\begin{equation}\label{si}
0\to\sha^{\e i}(F,M\le)\to\sha^{\e i}_{\omega}(F,M\le)\overset{\lambda^{\be i}}
\lra\displaystyle\bigoplus_{\text{all $v$}}H^{\e
i}(F_{\be v},M\le)\to\tha^{\le i}(F,M\le)\to 0,
\end{equation}
where $\lambda^{\lbe i}=\lambda^{\lbe i}(F,M\le)$ and $\tha^{i}(F, M\le)=\cok\lambda^{\lbe i}(F,M\le)$. Now set 
\begin{equation}\label{col}
\sha^{\le i}_{\omega}(K/F, M\e)=H^{\le i}(K/F, M\e)\cap \sha^{\le i}_{\omega}(F,M\e),
\end{equation}
where $H^{\le i}(K/F, M\e)$ is the group \eqref{hlg} and the intersection takes place inside $H^{\le i}(F, M\e)$. 

We define a map 
\begin{equation}\label{l1k}
\lambda^{\lbe i}_{\lbe K\be/\lbe F}=\lambda^{\lbe i}(K/F, M\e)\colon \sha^{\le i}_{\omega}(K/F, M\e)\to \bigoplus_{\text{all $v$}}H^{\e i}(\kwvf\le,M\e)
\end{equation}
and groups
\begin{equation}\label{l1k-0}
\sha^{\e i}\lbe(K/F,M\e)=\krn \lambda^{\lbe i}_{\lbe K\be/\lbe F}
\end{equation}
and
\begin{equation}\label{l1k-1}
\tha^{\le i}\lbe(K/F,M\e)=\cok \lambda^{\lbe i}_{\lbe K\be/\lbe F}
\end{equation}
so that the following diagram, whose bottom row is the exact sequence \eqref{si}, is exact and commutative:
\begin{equation}\label{4}
\xymatrix{\sha^{\le i}\lbe(K/F,M\e)\,\ar@{^{(}->}[d]\ar@{^{(}->}[r]& \sha^{\le i}_{\omega}(K/F, M\e)\ar@{^{(}->}[d]\ar[r]^(.415){\lambda^{\be i}_{\lbe K\be/\be F}}& \displaystyle\bigoplus_{\text{all $v$}}H^{\le i}(\kwvf\le,M\e)\ar@{^{(}->}[d]\ar@{->>}[r]& \tha^{\le i}\lbe(K/F,M\e)\ar[d]\\
\sha^{\le i}\lbe(F,M\e)\,\ar@{^{(}->}[r]& \sha^{\le i}_{\omega}(F,M\e)\ar[r]\ar[r]^(.45){\lambda^{\be i}}& \displaystyle\bigoplus_{\text{all $v$}}H^{\le i}(F_{\be v},M\e)\ar@{->>}[r]& \tha^{\le i}\lbe(F,M\e).}
\end{equation}

\begin{lemma}\label{shx} The groups $\sha^{\le i}\lbe(K/F,M\e)$ \eqref{l1k-0} and $\tha^{\le i}\lbe (K/F,M\e)$ \eqref{l1k-1} are annihilated by $[\e K\lbe\colon \! F\,]$.
\end{lemma}
\begin{proof} By Lemmas \ref{hi} and \ref{lexp}\,, respectively, the source \eqref{col} and the target of the map $\lambda^{\lbe i}(K/F, M\e)$ \eqref{l1k} are annihilated by $[\e K\lbe\colon \! F\,]$, which immediately yields the lemma.
\end{proof}

We now observe that the kernel of $\res_{\lbe\omega}^{\le i}\colon \sha^{\le i}_{\omega}(F,M\e)\to \sha^{\le i}_{\omega}(K,M\e)^{\g}$ equals $H^{\le i}(K/F, M\e)\cap \sha^{\le i}_{\omega}(F,M\e)=\sha^{\le i}_{\omega}(K/F, M\e)$ \eqref{col}. Restricting the preceding map to $\sha^{\le i}(F,M\e)\subset \sha^{\le i}_{\omega}(F,M\e)$, we obtain a map
\begin{equation}\label{r2}
\widetilde{\res}^{\le i}\,\colon\, \sha^{\le i}(F,M\e)\to \sha^{\le i}\lbe(K,M\e)^{\g}\be\cap \res_{\lbe\omega}^{\le i}\le(\sha^{\le i}_{\omega}(F,M\e))
\end{equation}
such that the composite map
\begin{equation}\label{fact}
\sha^{\le i}(F,M\e)\,\overset{\!\widetilde{\res}^{\le i}}{\lra}\,\sha^{\le i}\lbe(K,M\e)^{\g}\be\cap \res_{\lbe\omega}^{\le i}\le(\sha^{\le i}_{\omega}(F,M\e))\hookrightarrow \sha^{\le i}\lbe(K,M\e)^{\g}
\end{equation}
equals $\res^{\le i}\colon\, \sha^{\le i}(F,M\e)\to \sha^{\le i}\lbe(K,M\e)^{\g}$.

Next, we introduce the group
\begin{equation}\label{tr}
\tra\!\left(\sha^{\le i}\be(K,M\e)^{\g}\e\right)=\displaystyle\frac{\sha^{\le i}\lbe(K,M\e)^{\g}}{\sha^{\le i}\lbe(K,M\e)^{\g}\be\cap \res_{\lbe\omega}^{\le i}\le(\sha^{\le i}_{\omega}(F,M\e))}.
\end{equation}
Note that
\[
\tra\!\left(\sha^{\le i}\be(K,M\e)^{\g}\e\right)
\simeq\displaystyle\frac{\sha^{\le i}\lbe(K,M\e)^{\g}+\res_{\lbe\omega}^{\le i}\le(\sha^{\le i}_{\omega}(F,M\e))}{\res_{\lbe\omega}^{\le i}\le(\sha^{\le i}_{\omega}(F,M\e))}\subset {\rm C}_{\omega}^{\e i}\lbe(K/F, M\e),
\]
where ${\rm C}_{\omega}^{\le i}\lbe(K/F, M\e)$ is given by \eqref{sci}. Since $\sha^{\le i}\be(K,M\e)(\e\ell\e)$ is a group of finite cotype for every prime $\ell$ by Lemma \ref{fcot} and ${\rm C}_{\omega}^{\le i}\lbe(K/F, M\e)$ is a group of finite exponent annihilated by a power of $[\e K\lbe\colon\lbe \! F\,]$ by Lemma \ref{sfex}\,, \eqref{tr} is a {\it finite} group annihilated by a power of $[\e K\lbe\colon\lbe \! F\,]$.

\begin{remark}\label{nwt} Our choice of notation in \eqref{tr} is motivated by the fact that, if $G$ is a semiabelian variety over $F$, then
\[
\tra\!\left(\sha^{\le 1}\be(K,G\e)^{\g}\e\right)\simeq\trans\!\left(\sha^{\le 1}\be(K,G\e)^{\g}\e\right)\subset  H^{\e 2}(\g, G(K\le)),
\]
where $\trans\colon H^{\e 1}(K, G\le)^{\g}\!\be\to\be H^{\e 2}(\g, G(K\le))$ is the transgression map \cite[p.~51]{sha}.
\end{remark}

We now apply Proposition \ref{ker-cok} to the pair of maps \eqref{fact} and obtain an equality
\begin{equation}\label{kk}
\krn\be(\res^{\le i})=\krn\be(\widetilde{\res}^{\le i})
\end{equation}
and a canonical exact sequence
\begin{equation}\label{pp}
0\to\cok\be(\widetilde{\res}^{\le i})\to\cok\be(\res^{\le i})\to \tra\!\left(\sha^{\le i}\be(K,M\e)^{\g}\e\right)\to 0.
\end{equation}
It is now clear that Theorem {\rm \ref{main1}} is true if, and only if, the kernel and cokernel of the map $\widetilde{\res}^{\le i}$ \eqref{r2} are finite groups annihilated by a power of $[\e K\lbe\colon\lbe \! F\,]$.

By the semilocal theory of \cite[\S2.1]{ch}, there exist isomorphisms of abelian groups
\begin{equation}\label{comp}
\left(\,\textstyle\bigoplus_{\e\text{all $w^{\prime}$}}H^{\le i}(K_{w^{\prime}},M\e)\right)^{\be\g}\simeq\textstyle\bigoplus_{\e\text{all $v$}}\!\left(\,\bigoplus_{\e w^{\prime}|v}H^{\le i}(K_{w^{\prime}},M\e)\right)^{\be\g}\simeq\textstyle\bigoplus_{\e\text{all $v$}}\e H^{\le i}(\kwv,M\e)^{\gv}
\end{equation}
where, in the last direct sum, for each prime $v$ of $F$, $w$ is the fixed prime of $K$ lying above $v$ chosen previously. We define a map
\begin{equation}\label{rlg}
\lambda_{\lbe K}^{\lbe i}\colon \res_{\lbe\omega}^{\le i}\le(\sha^{\le i}_{\omega}(F,M\e))\to \bigoplus_{\text{all$\e v$}}\e H^{\le i}(\kwv,M\e)^{\gv}
\end{equation}
by the commutativity of the diagram
\[
\xymatrix{\sha^{\le i}_{\omega}(K, M\e)^{\g}\ar[rr]^(.42){\lambda^{\lbe i}\be(\lbe K, M\e)^{\lbe\g}}&&\left(\,\e\displaystyle\bigoplus_{\text{all $w^{\prime}$}}H^{\le i}(K_{w^{\prime}},M\e)\right)^{\!\!\g}\ar[d]^{\wr}\\
\res_{\lbe\omega}^{\le i}\le(\sha^{\le i}_{\omega}(F,M\e))\ar@{^{(}->}[u]\ar[rr]^(.42){\lambda_{\lbe K}^{\be  i}}&&\displaystyle\bigoplus_{\text{all $v$}}\e H^{\le i}(\kwv,M\e)^{\gv},}
\]
where $\lambda^{\lbe i}\be(\lbe K, M\e)$ is the map \eqref{loc} associated to $M_{K}=[Y_{K}\to G_{K}]$ and the right-hand vertical map is the composition of the isomorphisms \eqref{comp}. Note that, since the kernel of $\lambda^{\lbe i}\be(\lbe K, M\e)$ equals $\sha^{\le i}\lbe(K,M\e)^{\g}$ by \eqref{si} over $K$, we have
\begin{equation}\label{ker}
\krn \lambda_{\lbe K}^{\lbe i}=\sha^{\le i}\lbe(K,M\e)^{\g}\be\cap \res_{\lbe\omega}^{\le i}\le(\sha^{\le i}_{\omega}(F,M\e)).
\end{equation}
The maps \eqref{loc}, \eqref{l1k} and \eqref{rlg} fit into an exact and commutative diagram of abelian groups
\[
\xymatrix{0\ar[r]& \sha^{\le i}_{\omega}(K/F, M\e)\ar[d]^{\lambda^{\lbe i}\lbe(\lbe K\lbe/\lbe F,M\e)}\ar[r]& \sha^{\le i}_{\omega}(F,M\e)\ar[d]^{\lambda^{\lbe i}\lbe(F,M\e)}\ar[r]& \res_{\lbe\omega}^{\le i}(\sha^{\le i}_{\omega}(F,M\e))\ar[d]^{\lambda_{\lbe K}^{\lbe i}}\ar[r]&0\\
0\ar[r]& \displaystyle\bigoplus_{\text{all $v$}}H^{\le i}(\kwvf\le,M\e)\ar[r]& \displaystyle\bigoplus_{\text{all $v$}}H^{\le i}(F_{\be v},M\e)\ar[r]&\displaystyle\bigoplus_{\text{all $v$}}H^{\le i}(\kwv,M\e)^{\gv}.}
\]
Now we apply the snake lemma to the preceding diagram, using \eqref{ker} and the definitions \eqref{l1k-0} and \eqref{l1k-1}, and obtain an exact sequence of discrete and torsion abelian groups
\[
0\to \sha^{\le i}\lbe(K/F,M\e)\to \sha^{\le i}\lbe(F,M\e)\overset{\!\widetilde{\res}^{\le i}}{\lra} \sha^{\le i}\lbe(K,M\e)^{\g}\e\cap\, \res\e(\sha^{\le i}_{\omega}(F,M\e))\to \tha^{\le i}\lbe(K/F,M\e)
\]
Since $\sha^{\le i}\lbe(F,M\e)(\e\ell\e)$ and $\sha^{\le i}\lbe(K,M\e)(\e\ell\e)$ are groups of finite cotype for every prime $\ell$ by Lemma \ref{fcot}\,, and  $\sha^{\le i}\lbe(K/F,M\e)$ and $\tha^{i}(K/F,M\e)$ are groups of finite exponent annihilated by $[\e K\lbe\colon \! F\,]$ by Lemma \ref{shx}\,, the preceding sequence together with Lemma \ref{es2} show that both $\krn\lbe(\widetilde{\res}^{\le i})$ and $\cok\lbe(\widetilde{\res}^{\le i})$ are finite groups annihilated by $[\e K\lbe\colon \! F\,]$. The proof of Theorem \ref{main1} is now complete.

\begin{remark} \label{last} It follows from the above and \eqref{kk} that $\krn\!(\res^{\le i})$ is a finite group annihilated by $[\e K\lbe\colon \! F\,]$ for $i=1$ and $2$. Further, if $M=G$ is a semiabelian variety, then \eqref{pp} and Remark \ref{nwt}\,  show that the cokernel of $\res^{\le 1}\colon \sha^{\le 1}\lbe(F,G\e)\to \sha^{\le 1}\lbe(K,G\e)^{\g}$ is annihilated by $[\e K\lbe\colon \! F\,]^{\e 2}$. 
\end{remark}

\section{Appendix}\label{app}

Since \cite{ga09} is cited in this paper, we repair here a (fortunately inconsequential) error which appears in \cite[proof of Proposition 4.10, p.~217, lines -8 to -3]{ga09}. The notation is that of [op.cit.].

Recall the element $(\xi_{\le v})\in \bigoplus_{\e v\notin U}H^{\le i}(K_{v},N\le)\times\prod_{\e v\in U}H^{\le i}(\mathcal O_{v},\s N\e)$, where the set $U$ has the property that $D^{\le 1}(V, \s N\le)=\sha^{1}(K,N\e)$ for any nonempty open subset $V$ of $U$. There exists an element $\xi_{\le U}\be\in H^{\e i}(U,\s N\e)$ such that $\xi_{\le U}|_{K_{v}}=\xi_{v}$ for all $v\notin U$. The incorrect part of the proof starts with the statement ``The assignment $(\xi_{v})_{v\notin U}\mapsto\xi_{\e U}$ is functorial in $U$" and ends with the claim ``This shows that $\krn\gamma_{i}(K,N)\subset\img\beta_{i}(K,N)$". The problem is that the indicated assignment is not even well-defined since the element $\xi_{\e U}\in H^{i}(U,\s N\e)$ is not uniquely determined. Thus, contrary to what is stated in [loc.cit.], it is not possible to obtain a well-defined element $\xi\in H^{\e i}(K,N)$ such that $\beta_{\e i}(K,N\e)(\e\xi\e)=(\xi_{v})$. This problem can be repaired as follows. The element $\xi_{\e U}\in H^{i}(U,\s N\e)$ is well-defined modulo $D^{\le i}(U, \s N\le)$, and the problem reduces to showing that the elements 
$\xi_{\e U}+D^{\le i}(U, \s N\e)\in H^{i}(U,\s N\e)/D^{\le i}(U, \s N\le)$ can be chosen so that, if $V$ is a nonempty open subset of $U$, then $\xi_{\e U}|_{V}-\xi_{V}\in D^{\le i}(V, \s N\e)$. Indeed, if this is proved, then any representative $\xi\in H^{i}(K,N\e)$ of the resulting well-defined class in $H^{i}(K,N\e)/\e\sha^{i}(K,N\e)$ satisfies $\beta_{\e i}(K,N\e)(\e\xi\e)=(\xi_{v})$ since $\sha^{i}(K,N\e)=\krn\beta_{\e i}(K,N\e)$. Now, if $i=1$ and $v\in U-V$, then $\xi_{\le V}|_{K_{v}}=\xi_{v}\in H^{\le 1}(\mathcal O_{v},\s N\e)$, whence $\xi_{\le V}\in H^{1}(U,\s N\e)\subset H^{1}(V,\s N\e)$ by \cite[Corollary 4.2 and Remark 4.3]{ces3}. Since $(\xi_{\le U}-\xi_{V})|_{K_{v}}=0$ for all $v\notin U$, we have $\xi_{\le U}-\xi_{V}\in D^{\le 1}(U, \s N\le)=\sha^{\le 1}(K,N\le)$. Consequently $\xi_{\e U}|_{V}-\xi_{V}\in \sha^{\le 1}(K,N\le)=D^{\le 1}(V, \s N\le)$, as required. If $i=2$ and $v\in U-V$, we have $\xi_{\e U}|_{K_{v}}=(\xi_{\e U}|_{\mathcal O_{v}})|_{K_{v}}=\xi_{v}=0$ since $H^{\le 2}(\mathcal O_{v},\s N\e)=0$ for all $v\in U$, whence $(\xi_{\e U}|_{V}-\xi_{V})|_{K_{v}}=0$ for all $v\notin V$, i.e., $\xi_{\e U}|_{V}-\xi_{V}\in D^{\le 2}(V, \s N\le)$.

\end{document}